\documentclass[12pt]{article}
\usepackage{bbm}
\usepackage{latexsym}
\usepackage{mathrsfs}
\usepackage{graphicx}
\usepackage{subfigure}
\usepackage{amssymb}
\usepackage{amsmath}
\usepackage{amsthm}
\usepackage{color}
\usepackage{cite}
\usepackage{float}
\usepackage{indentfirst}
\usepackage{anysize}\marginsize{45mm}{45mm}{40mm}{50mm}
\usepackage{caption}
\usepackage{float}

\usepackage{enumerate}

\newtheoremstyle{lemma}{\topsep}{\topsep}%
     {}
     {}
     {\bfseries}
     {}
     {0.1em}
     {\thmname{#1}\thmnumber{ #2}\thmnote{ #3}}
\theoremstyle{lemma}  

\newtheorem{theorem}{Theorem}[section]    
\newtheorem{lemma}[theorem]{Lemma}
\newtheorem{corollary}[theorem]{Corollary}

\numberwithin{equation}{section}

\usepackage{latexsym, bm}
\title
{Extremal anti-forcing numbers of perfect matchings of graphs\thanks{Supported by NSFC (grant no. 11371180 and 11401279).} }
\author{{Kai Deng$^{a,b}$, \  Heping Zhang$^{a,}$\thanks{
Corresponding author. \newline
\emph{E-mail address}: dengkai04@126.com (K. Deng), zhanghp@lzu.edu.cn (H. Zhang). }}\\
{\footnotesize $^{a}$School of Mathematics and Statistics, Lanzhou University, Lanzhou, Gansu 730000, P. R.~China}\\
{\footnotesize $^{b}$School of Mathematics and Information Science, Beifang University of Nationalities,}\\{\footnotesize Yinchuan, Ningxia 750027, P. R.~China}}
\date{}
\begin{document}

\maketitle
\begin{abstract}
The anti-forcing number of a perfect matching $M$ of a graph $G$ is the minimal number of edges not in $M$ whose removal to make $M$ as a unique perfect matching of the resulting graph. The set of anti-forcing numbers of all perfect matchings of $G$ is the anti-forcing spectrum of $G$. In this paper, we characterize the plane elementary bipartite graph whose minimum anti-forcing number is one.
We show that the maximum  anti-forcing number of a  graph is at most its cyclomatic number. In particular, we  characterize the  graphs with the maximum anti-forcing number achieving the upper bound, such extremal graphs are a class of plane bipartite graphs. Finally, we determine the anti-forcing spectrum of an even polygonal chain in linear time.

\vskip 0.1 in


\end{abstract}

\section{Introduction}

We only consider finite and simple graphs. Let $G$ be a graph with vertex set $V(G)$ and edge set $E(G)$.
A {\em perfect matching} or 1-factor
of $G$ is a set of disjoint edges which covers all vertices of $G$.
A perfect matching of a graph coincides with a Kekul\'{e} structure in organic chemistry and a dimer in statistic physics.

The concept of ``forcing" has been used in many research fields in graph theory and combinatorics \cite{Che,Mahmoodian}. It appeared first in a perfect matching $M$ of a graph $G$ due to Harary et al. \cite{Harary}: If a subset $S$ of $M$ is not contained in other perfect matchings of $G$, then we say $S$ {\em forces} the  perfect matching $M$,
in other words, $S$ is called a \emph{forcing set} of $M$.
The minimum cardinality over all forcing sets of $M$
is  called the \emph{forcing number} of $M$.
The roots of those concepts can be found in an earlier chemical
literature due to Randi\'{c} and Klein  \cite{Klein},
under the name of  the \emph{innate degree of freedom} of a  Kekul\'{e} structure,
which plays an important role in the resonance theory in chemistry. In 1990's Zhang and Li \cite{Li} and Hansen and Zheng \cite{Hansen} determined independently the hexagonal systems with a forcing edge.
Afterwards Zhang and Zhang \cite{Zhang} characterized plane elementary bipartite graphs with a forcing edge. For more researches on matching forcing problems,
see \cite{Adams,Afshani,Jiang,JZ16,Pachter,Matthew,Wang,Ye,ZhangD}.


 Vuki\v{c}evi\'{c} and Trinajsti\'{c} \cite{VT1} introduced the \emph{anti-forcing number} of a graph $G$
as the smallest number of edges
whose removal results in a subgraph with a unique perfect matching,
denoted by $af(G)$.
So
a graph $G$ has a unique perfect matching if and only if its anti-forcing number is zero.
An edge $e$ of a graph $G$ is called an \emph{anti-forcing edge} if $G-e$ has a unique perfect matching.
As early as  1997 Li \cite{Li1}  showed that the hexagonal systems with an anti-forcing edge (under the name ``forcing single edge'') are truncated parallelograms.
Deng \cite{Deng1} gave a linear time algorithm to compute the anti-forcing number of benzenoid chains. Yang et al. \cite{YZL} showed that a fullerene has the anti-forcing number at least four.  For other works, see \cite{VT2,  Deng2, ZhangQQ}.

Recently, Lei et al. \cite{Zh2} defined the anti-forcing number of a perfect matching $M$ of a graph $G$
as the minimal number of edges not in $M$ whose removal to make $M$ as a single perfect matching of the resulting graph, denoted by $af(G,M)$.
By this definition,
$af(G)$ is the smallest anti-forcing number over all perfect matchings
of $G$,  called the \emph{minimum anti-forcing number} of $G$. Also let $Af(G)$ denote  the largest anti-forcing number over all perfect matchings
of $G$, called the \emph{maximum anti-forcing number} of $G$. They also showed that the maximum anti-forcing number of a hexagonal system equals its fries number. The present authors \cite{DZ} considered the \emph{anti-forcing spectrum} $\text{Spec}_{af}(G)$ as the set of anti-forcing
numbers of perfect matchings in $G$.

Let $M$ be a perfect matching of a graph $G$.
A subset $S\subseteq E(G)\backslash M$ is called an \emph{anti-forcing set} of $M$ if $M$ is the unique perfect matching of $G-S$.
A cycle $C$ of $G$ is called an \emph{$M$-alternating cycle} if  the edges of $C$ appear alternately in $M$ and $E(G)\backslash M$.
If $C$ is an $M$-alternating cycle of $G$, then the symmetric difference $M\triangle C:=(M-C)\cup(C-M)$ is another perfect matching of $G$. In this case a cycle  is always regarded as its edge set.

\begin{theorem}\label{antiforcingset}\cite{Zh2}{\bf .}
An edge set $S$ of a graph $G$ is an anti-forcing set of a perfect matching $M$ of $G$ if and only if
$S$ contains at least one edge of every $M$-alternating cycle of $G$.
\end{theorem}

Let $M$ be a perfect matching of a graph $G$.
A set $\mathcal{A}$ of $M$-alternating cycles of $G$ is called a \emph{compatible $M$-alternating set} if any two members of $\mathcal{A}$
either are disjoint or intersect only at edges in $M$. Let $c^{\prime}(M)$ denote the cardinality of a maximum compatible $M$-alternating set
of $G$. By Theorem \ref{antiforcingset},
we have $af(G,M)\geq c^{\prime}(M)$.
For plane bipartite graphs, the equality holds.

\begin{theorem}\label{compatibleset}\cite{Zh2}{\bf .}
Let $G$ be a plane bipartite graph with a perfect matching $M$. Then $af(G,M)=c^{\prime}(M)$.
\end{theorem}

Throughout this paper all the bipartite graphs are given a proper
black and white coloring:  any two adjacent vertices receive different colors.
An edge of a graph  $G$ is  \emph{allowed} if it belongs to a perfect matching of $G$
and \emph{forbidden} otherwise.
$G$ is said to be \emph{elementary} if all its
allowed edges form a connected subgraph of $G$. It is well-known that a connected bipartite graph is elementary if and only if each edge is allowed \cite{Plummer}.

An elementary bipartite graph has the so-called ``\emph{bipartite ear decomposition}".
Let $x$ be an edge.
Join the end vertices of $x$
by a path $P_1$ of odd length (the so-called ``first ear").
We proceed inductively to build a sequence of bipartite graphs as follows: If $G_{r-1}=x+P_{1}+P_{2}+\cdots+P_{r-1}$
has already been constructed, add the $r$-th ear $P_r$ (a path of odd length) by joining any two vertices in different colors of $G_{r-1}$
such that $P_r$ has no other vertices in common with $G_{r-1}$.
The decomposition $G_{r}=x+P_{1}+P_{2}+\cdots+P_{r}$ will be called a bipartite ear decomposition of $G_r$.

\begin{theorem}\label{eardecomposition}\cite{Plummer1}{\bf .}
A bipartite graph is elementary if and only if it has a bipartite ear decomposition.
\end{theorem}

A bipartite ear decomposition  $G =x+P_{1}+P_{2}+\cdots+P_{r}$
can be represented   by a sequence of  graphs ($G_0,G_1,\ldots,G_r(=G)$),
where $G_0=x$ and $G_{i}=G_{i-1}+P_{i}$ for $1\leq i\leq r$.
We can see that the number of ears equals $|E(G)|-|V(G)|+1$,
i.e., the \emph{cyclomatic number} of $G$, denoted by $r(G)$.

A bipartite ear decomposition ($G_1(=x+P_1),\ldots,G_r(=G)$) of
a plane elementary bipartite graph $G$ is called
a \emph{reducible face decomposition} if $G_1$ is the boundary of an interior
face of $G$ and the $i$-th ear $P_i$ lies in the exterior of $G_{i-1}$ such that $P_i$
and the part of the periphery of $G_{i-1}$ bound an interior face of $G$ for all $2\leq i\leq r$.

\begin{theorem}\label{face}\cite{Zhang}{\bf .}
Let $G$ be a plane  bipartite graph other than $K_2$.
Then $G$ is elementary if and only if $G$ has a reducible face decomposition starting with the boundary of any interior face of $G$.
\end{theorem}

In the next section, we characterize the plane elementary bipartite graphs with anti-forcing edges by using reducible face decomposition.
In section 3,
we show that the maximum  anti-forcing number
of a connected graph with a perfect matching
is at most its cyclomatic number. In particular we characterize the graphs
with the maximum anti-forcing number achieving this cyclomatic number in terms of bipartite ear decomposition. We shall see that such extremal graphs are a special type of plane bipartite graphs, and have a unique perfect matching whose anti-forcing number is maximum.
In Section 4, we show that an even polygonal chain including benzenoid chain has the continuous anti-forcing spectrum.
So we can determine the  anti-forcing spectrum by  designing  linear algorithms to compute the minimum and
maximum anti-forcing numbers of an even polygonal chain.

\section{Anti-forcing edge}

The \emph{$Z$-transformation graph} $Z(G)$ of a plane bipartite graph $G$ is defined as the graph whose vertices represent the perfect matchings
of $G$ where two vertices are adjacent if and only if the symmetric difference of the corresponding two perfect matchings just forms the
boundary of an interior face of $G$.
A face of $G$ is said to be \emph{resonance} if its boundary is
an $M$-alternating cycle
with respect to a perfect matching $M$ of $G$.
By using reducible face decomposition,
Zhang and Zhang \cite{Zhang} described those plane elementary bipartite graphs whose $Z$-transformation graphs have a vertex of degree one
and characterized the plane elementary bipartite graphs with a forcing edge.

\begin{theorem}\label{forcingedge} \cite{Zhang}{\bf .}
A plane elementary bipartite graph $G$ has a forcing edge if and only if $G$ has a perfect matching $M$ such that
 $G$ has exactly two $M$-resonance faces (the exterior face is allowed) and their boundaries are intersecting.
 Further each common edge in $M$ on the two $M$-resonance faces is a forcing edge of $G$.
\end{theorem}

In the following, we characterize plane elementary bipartite graphs
with an anti-forcing edge.
The following lemma is useful.

\begin{lemma}\label{onedegree}\cite{Plummer}{\bf .}
Let $G$ be a bipartite graph with a unique perfect matching.
Then $G$ must contain at least one vertex of degree 1 in each color class.
\end{lemma}

By Theorem \ref{antiforcingset}, the following result is immediate.

\begin{lemma}\label{pass}{\bf .}
Let $G$ be a graph with a perfect matching $M$. Then $e\in E(G)\setminus M$ is an anti-forcing edge
if and only if each $M$-alternating cycle passes through $e$.
\end{lemma}

\begin{theorem}\label{antiforcingedge}{\bf .}
A plane elementary bipartite graph $G$ has an anti-forcing edge if and only if $G$ has a perfect matching $M$ such that
 $G$ has exactly two $M$-resonance faces whose boundaries have a common path with length at least 3.
\end{theorem}

\begin{proof}
 Let $e=uv$ be an anti-forcing edge of $G$. Then $G-e$ has a unique perfect matching $M$  (we may say $M$ is anti-forced by $e$).
By Lemma \ref{onedegree}, $G-e$ has at least two vertices of degree 1.
Since $G$ is 2-connected,  only $u$ and $v$ are of degree 1 in $G-e$.
Let  $f$ (resp. $g$) be the edge which is incident to $u$ (resp. $v$) in $G-e$.
Then $f\in M$ (resp. $g\in M$) is a forcing edge of $G$.
By Theorem \ref{forcingedge},
$G$ has exactly two $M$-resonant faces  and their boundaries $s_1$ and $s_2$ are intersecting.
By Lemma  \ref{pass}, $s_1$ and $s_2$ both pass through $e$.
Since $u$ and $v$ both are  of degree 2 in $G$,  three edges $f,e$ and $g$ form a path with length at least 3 lying on both $s_1$ and $s_2$.

Conversely, suppose  for a perfect matching $M$ of $G$ there are exactly two $M$-resonance faces  of $G$ whose boundaries  have a common path $P$ with length at least 3
 Then $P$ has a pair of adjacent edges $e$ and $f$ such that $e\notin M$ and $f\in M$.
  By Theorem \ref{forcingedge}, $f$ is a forcing edge of $G$. Note that a perfect matching $M$ of $G-e$ contains $f$. Hence $M$ is a unique perfect matching of $G-e$, and thus
  $e$ is an anti-forcing edge of $G$.
\end{proof}


From Theorem \ref{antiforcingedge} a plane elementary bipartite graph with an anti-forcing edge can be given an ear construction. For example, see Fig. \ref{ear}. According to Theorems \ref{forcingedge} and \ref{antiforcingedge},
we have the following corollary.

\begin{corollary}\label{fanti}{\bf .}
Let $G$ be a plane elementary bipartite graph with an anti-forcing edge.
Then  $G$ must have a forcing edge.
\end{corollary}


\begin{figure}[http]
  \centering
    \includegraphics[width=40mm]{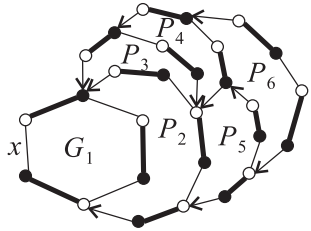}
       \caption{An ear  construction for a plane bipartite graph with an anti-forcing edge $x$.}
  \label{ear}
\end{figure}
%
%
%

\section{Maximum anti-forcing number}

\begin{theorem}\label{maxantiforcing}{\bf .}
Let $G$ be a connected graph with a perfect matching.
Then $Af(G)\leq r(G)$.
Further, if $G$ is nonbipartite,
then  $Af(G)<r(G)$.
\end{theorem}

\begin{proof}
Let $M$ be any perfect matching of $G$.
If $G$ has a cycle $C_1$,
then there is an edge $e_1\in E(C_1)\setminus M$
such that $G-e_1$ is connected and $M$ is a perfect matching of $G-e_1$.
If $G-e_1$ also contains a cycle $C_2$,
then we can delete an edge $e_2\in E(C_2)\setminus M$ from $G-e_1$
such that $G-e_1-e_2$ is connected and $M$ is a perfect matching of $G-e_1-e_2$.
Repeating this procedure,
finally we can obtain a spanning tree $T$ of $G$ such that
$M$ is a unique perfect matching of $T$.
So $E(G)\setminus E(T)$ is an anti-forcing set of $M$ in $G$,
and $af(G,M)\leq |E(G)\setminus E(T)|=r(G)$.
It implies that $Af(G)\leq r(G)$.

Since  tree $T$ is a bipartite graph, its vertex set can be partitioned into
two partite sets $X$ and $Y$.
If $G$ is nonbipartite,
then there must be an edge $e$ of $E(G)\setminus E(T)$
such that both ends of $e$  are in $X$ or $Y$.
Therefore, $T+e$ contains a unique cycle, which is odd.
So $T+e$ has no $M$-alternating cycles. That is, $M$ is still a unique
perfect matching of $T+e$.
Hence $M$ is anti-forced by $E(G)\setminus E(T+e)$,
and  $af(G,M)\leq |E(G)\setminus E(T+e)|< r(G)$, which implies that $Af(G)<r(G)$.
\end{proof}

In the following we characterize graphs $G$ with $Af(G)=r(G)$.
By Theorem \ref{maxantiforcing},
we only need to consider bipartite graphs. Let $G$ be a  bipartite graph with a perfect matching.
An edge of $G$ is said
to be \emph{fixed single} (resp. \emph{double}) if it belongs to no (resp. all) perfect matchings of $G$.
An edge of $G$ is \emph{fixed}
if it is either fixed double edge or fixed single edge.
The non-fixed edges of $G$ form a subgraph of $G$, each component of which is elementary and  called \emph{elementary (or normal) component} of $G$. The proof of Lemma 4.5 in \cite{Zh2} also implies the below.

\begin{lemma}\label{sum}{\bf .} Let $G$ be a  bipartite graph with a perfect matching $M$. Let $G_1,G_2,\ldots,G_k$ be the normal components of $G$, $k\geq 1$.
Then
$af(G,M)=\sum^{k}_{i=1}af(G_i,M_i)$, where $M_i$ is the restriction of $M$ on $G_i$.
\end{lemma}

\begin{lemma}\label{general}{\bf .}
Let $G$ be a  bipartite graph with a perfect matching. Then $Af(G)=r(G)$ if and only if  each fixed edge of $G$
is a cut-edge and for each normal component of $G$ the maximum anti-forcing numbers is equal to its cyclomatic number.
\end{lemma}

\begin{proof}If $G$ has no normal components, then $Af(G)=0$. In this case $Af(G)=r(G)$ if and only if $G$ is a forest.  Adapting the notations in Lemma \ref{sum}, we assume that $k\geq 1$. Lemma \ref{sum} implies that $Af(G)=\sum^{k}_{i=1}Af(G_i)$. By Theorem \ref{maxantiforcing}, we have

$$Af(G)=
\sum^{k}_{i=1}Af(G_i)\leq \sum^{k}_{i=1}r(G_i)\le  r(G).$$

\noindent So $Af(G)=r(G)$ if and only if $Af(G_i)=r(G_i)$ for each $1\le i\le k$, and $\sum^{k}_{i=1}r(G_i)=  r(G)$. The latter holds if and only if each fixed edge of $G$ is a cut edge.
\end{proof}


From Lemma \ref{general}  we need to consider
elementary  bipartite graphs.

\begin{lemma}\label{cha}{\bf .}
Let $G$ be an elementary  bipartite graph other than $K_2$.
Then $Af(G)=r(G)$ if and only if $G$ has
a perfect matching $M$ and
a bipartite ear decomposition $(G_{0},G_{1},\ldots,G_{r}(=G))$, $r:=r(G)\geq 1$,
such that  $M_i=M\cap E(G_i)$
is a perfect matching of $G_i$, $0\leq i\leq r$,
and $u_iv_i\in M_{i-1}$ for the  ends $u_i$ and $v_i$
of the $i$-th ear $P_i$, $1\le i\le r$.
\end{lemma}

\begin{proof}
\emph{Sufficiency}:
Let $C_i=P_i+u_iv_i$ ($1\leq i\leq r$).
Then $\{C_1,C_2,\ldots,C_r\}$ is a compatible
$M$-alternating set of $G$.
By Theorem  \ref{antiforcingset},
$af(G,M)\geq r$.
By Theorem \ref{maxantiforcing},
$af(G,M)\leq r$.
So $af(G,M)=r$.
It implies that $Af(G)=r$.

\emph{Necessity}:
Let $M$ be a perfect matching of $G$ with $af(G,M)=r$.
From the proof of Theorem \ref{maxantiforcing},
$G$ has a spanning tree $T$  such that $M$ is
a unique perfect matching of $T$.
Let $S=\{e_1,e_2,\ldots,e_r\}:=E(G)\setminus E(T)$,
and $C_i$ the unique cycle in $T+e_i$ for $i=1,2,\ldots,r$. Let $\mathcal{A}=\{C_1,C_2,\ldots,C_r\}$. Then $\mathcal{A}$ is a system of fundamental cycles of $G$ relative to $T$, which is a base of the cycle space of $G$.
\vskip 4mm

\noindent\textbf{Claim 1.} $\mathcal{A}$ is a compatible $M$-alternating set.

\begin{proof}We have that each $C_i$ is $M$-alternating.
Otherwise, there is a cycle $C_j$ ($1\leq j \leq r$) which is not $M$-alternating. So $M$ is the unique perfect matching of
$T+e_j$, and $S\setminus\{e_j\}$ is an anti-forcing set of $M$.
It implies that $af(G,M)\leq |S\setminus\{e_j\}|=r-1$, a contradiction.

To the contrary, suppose $C_i,C_j\in \mathcal{A}$ are incompatible.
Then there exists an edge $e\in (E(C_i)\cap E(C_j))\setminus M$.
It follows that $T+e_i+e_j-e$ contains exactly  one cycle
and this cycle is not $M$-alternating. So  $M$  is a unique perfect
matching of  graph $T+e_i+e_j-e$.
Hence $(S\setminus\{e_i,e_j\})\cup\{e\}$ can be an anti-forcing
set of $M$ in $G$, and $af(G,M)\leq |(S\setminus\{e_i,e_j\})\cup\{e\}|=r-1$, a contradiction.
 \end{proof}

\noindent\textbf{Claim 2.}  $G=\bigcup^{r}_{i=1}C_i$.

\begin{proof}For any edge $e\in E(G)\setminus M$,
$G-e$ is connected since $G$ is 2-connected.
Suppose $e$ does not appear in any cycle of $\mathcal{A}$.
Then $\mathcal{A}$ also is a compatible $M$-alternating set
of $G-e$.
By Theorems \ref{antiforcingset} and \ref{maxantiforcing},
we have $r=|\mathcal{A}|\leq Af(G-e)\leq r(G-e)=r-1$, a contradiction.
For an edge $f\in M$,
any adjacent edge  of $f$ is not in $M$  and appears in an $M$-alternating cycle of $\mathcal{A}$,
so  $f$ appears in such a cycle.   \end{proof}

\noindent\textbf{Claim 3.} Any two cycles of $\mathcal{A}$ have at most one common edge in $M$.

\begin{proof}If  two cycles $C_i$ and $C_j$ in $\mathcal{A}$  have more than one common edges in $M$,
then such common edges belong to $M$ and are thus disjoint. So $C_i\cup C_j\subseteq T+e_i+e_j$ and has the cyclomatic number  at least three,
contradicting $r(T+e_i+e_j)=2$.\end{proof}

An edge of $M$ is called a \emph{shared edge} if
it belongs to at least two cycles of $\mathcal{A}$.
Since $G$ is connected,
each cycle of  $\mathcal{A}$ has at least one shared edge.
\vskip 4mm

\noindent\textbf{Claim 4.} $\mathcal{A}$ has a cycle that has exactly one shared edge.

\begin{proof}Suppose to the contrary that each cycle of $\mathcal{A}$ has at least
two common edges. Then $\mathcal{A}$ has a cyclic sequence of cycles $C_{i_1},C_{i_2},\ldots,C_{i_s}$, $s\geq 3$, such that just all pairs of consecutive $C_{i_j}$ and $C_{i_{j+1}}$ have one common edge for  $1\le j\le s$. This implies $r(\cup_{j=1}^{s}C_{i_j})=s+1$,  contradicting that $r(T+e_{i_1}+e_{i_2}+\cdots+e_{i_s})=s$.
\end{proof}

We now prove the necessity by induction on $r$.
For the case $r=1$, $G$ is an even cycle, so the result is trivial.
Suppose now that $r\geq 2$.

By Claim 4 let $C_r$ be a cycle of $\mathcal{A}$  that has exactly one shared edge $u_{r}v_{r}$. Then the vertices of $C_r$ except $u_r$ and $v_r$ are of degree 2.
Let $P_{r}:=C_{r}-u_{r}v_{r}$, and $G_{r-1}$ the graph obtained from $G$ by deleting the inner vertices of
the path $P_r$.
Then $M_{r-1}=M\cap E(G_{r-1})$ is a perfect matching of $G_{r-1}$,
and $u_{r}v_{r}\in M_{r-1}$.
Note that $G_{r-1}=\bigcup^{r-1}_{i=1}C_i$
is an elementary bipartite graph with $r(G_{r-1})=r-1\geq 1$,
and $\mathcal{A}\setminus\{C_r\}$
is a compatible $M_{r-1}$-alternating set of $G_{r-1}$.
Therefore  $Af(G_{r-1})=r-1$.
By the induction hypothesis,
$G_{r-1}$ has an ear decomposition
$(G_{0},G_{1},\ldots,G_{r-1})$
such that  $M_i=M_{r-1}\cap E(G_i)$
is a perfect matching of $G_i$ ($1\leq i\leq r-1$),
and the two end vertices $u_i$ and $v_i$
of the $i$-th ear $P_i$ are adjacent in $G_{i-1}$,
and $u_iv_i\in M_{i-1}$.
Adding $P_r$ to $G_{r-1}$ we obtain
the required ear decomposition
$(G_{0},G_{1},\ldots,G_{r-1}, G_{r})$ of $G$.
\end{proof}

From Lemma \ref{cha} we have that an elementary bipartite graph $G$ with $Af(G)=r(G)$ must be planar. This follows easily from the specific ear decomposition. So from the above lemmas we have the following immediate consequences.

\begin{theorem}\label{main}{\bf .}
Let $G$ be a graph $G$ with a perfect matching. Then
$Af(G)=r(G)$ if and only if $G$ is a planar bipartite graph and each block is either a fixed edge or a normal component with an ear decomposition as described in Lemma \ref{cha}.
\end{theorem}

\begin{corollary}\label{count}{\bf .}
Let $G$ be an elementary graph with $r(G)\geq 2$ and  $Af(G)=r(G)$.
Then $G$ has a unique perfect matching $M$ such that $af(G,M)=r(G)$.
\end{corollary}

\begin{proof} By Theorem \ref{main} we have that $G$ is a planar elementary bipartite graph.
We proceed by induction on $r(G)$.
For the case $r(G)=2$, by Lemma \ref{cha},
$G$ is the union of two even cycles
which have just one common edge $e$, and $e$ must belong to  perfect matching $M$ of $G$ with $af(G,M)=r(G)$.
The result holds from that $G$ has a unique perfect matching $M$ containing $e$.

We now consider the case $r:=r(G)\geq 3$. Because  $Af(G)=r(G)$,
for a perfect matching $M$ with $af(G,M)=r$, by Lemma \ref{cha} and its proof,
$G$ has an ear decomposition $(G_{1},G_{2},\ldots,G_{r}(=G))$
such that  $M_i=M\cap E(G_i)$
is a perfect matching of $G_i$,
and for the two ends $u_i$ and $v_i$
of the $i$-th ear $P_i$, $u_iv_i\in M_{i-1}$.
So $G=G_{r-1}+P_r$ and $af(G_{r-1},M_{r-1})=r-1$. For any (other) perfect matching $M^\prime$  of $G$
with $af(G,M^\prime)=r$, $G$ also has an ear decomposition $(G_{1}',G_{2}',\ldots,G_{r}'(=G))$ with the above properties. The other corresponding notations are given:  $P_i'$, $u_i'$ and $v'_i$, and $M_i'$.  We want to show $M=M^\prime$.

Let $C_i':=P_i'+u_i'v_i'$ for each $1\le i\le r$. Then by the proof of Lemma \ref{cha}  $\mathcal{A}^\prime=\{C_1^\prime,C_2^\prime,\ldots,C_r^\prime\}$ is a compatible $M'$-alternating set of $G$,
any two cycles of $\mathcal{A}^\prime$ have at most one
common edge in $M^\prime$,
and $G=\bigcup_{i=1}^{r}C_i^\prime$.
Hence $P_r$ is passed through by a cycle of $\mathcal{A}^\prime$.
Since $G$ is simple,
$P_r$ is a path with at least 3 edges.
So $P_r$ has an edge not in $M^\prime$. This
 implies that just one cycle of $\mathcal{A}^\prime$ can
pass through $P_r$, say $C_{i_0}^\prime$.
\vskip 4mm

\noindent {\bf Claim.}  $u_rv_r\in M^\prime$.

\begin{proof}
To the contrary,
suppose $u_rv_r\notin M^\prime$. Since $P_r\subset C_{i_0}'$, $P_r$ is an $M^\prime$-alternating path.
If both terminal edges of $P_r$  are not in $M^\prime$,
then $C_{i_0}^\prime$ has to pass the two
 edges $f$ and $g$ of $G_{r-1}$ which are in  $M'$ and incident to $u_r$ and $v_r$.
Since $G=\bigcup_{i=1}^{r}C_i^\prime$,
$u_rv_r$ is passed by a cycle
$C_{j_0}^\prime\in\mathcal{A}^\prime$ ($j_0\neq i_0$).
Noting that $C_{j_0}^\prime$ is $M^\prime$-alternating and $f,g\in M'$,
$C_{j_0}^\prime$ has to pass $f$ and $g$.
So $C_{i_0}^\prime$ and $C_{j_0}^\prime$ have at least two
common edges in $M^\prime$, a contradiction.
Now suppose both terminal edges of $P_r$  belong to
$M^\prime$. Then $C_{i_0}'=P_r+u_rv_r$. Let $f$ be an edge of $G_{r-1}$ adjacent to $u_rv_r$. Then $f\notin M'$ and $\mathcal{A}^\prime$ has one cycle $C_{i_0'}'$
which passes through $f$ and thus through $P_r$.  So $C_{i_0}^\prime$ and $C_{i_{0’}}^\prime$ both contain $P_r$. This implies that they have at least one common edge not in $M'$,  contradicting that $C_{i_0}^\prime$ and $C_{i_{0’}}^\prime$ are compatible.
\end{proof}

By the claim, $C_{i_0}'=P_r+u_rv_r$. so $G_{r-1}=\bigcup_{i=1, \not=i_0}^{r-1}C_i^\prime$, and
 $M'|_{G_{r-1}}:=M'\cap E(G_{r-1})$
is a perfect matching of $G_{r-1}$. Further $G_{r-1}$ has compatible $M'|_{G_{r-1}}$-alternating set $\mathcal{A}^\prime\setminus\{C_{i_0}'\}$  and compatible $M_{r-1}$-alternating set $\mathcal{A}\setminus\{C_r\}$, which implies that $af(G_{r-1},M'|_{G_{r-1}})\\=af(G_{r-1},M_{r-1})=r-1$.
By the induction hypothesis,
$M_{r-1}=M'|_{G_{r-1}}$.
On the other hand,  $M\cap E(P_r)=M^\prime\cap E(P_r)$. Hence $M=M'$. \end{proof}

\begin{figure}[h]
  \centering
    \includegraphics[width=35mm]{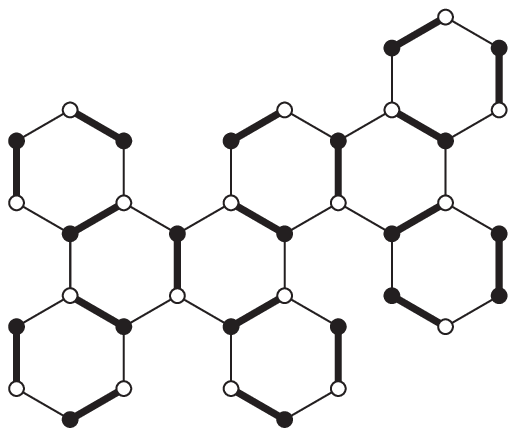}
    \hspace{10mm}
    \includegraphics[width=35mm]{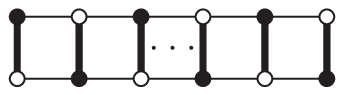}
       \caption{An all-kink catahex (left) and a straight chain polyomino (right).}
  \label{example}
\end{figure}

A \emph{polyomino} (resp. \emph{hexagonal system}) $G$ is a finite connected subgraph of a hexagonal (resp. square) grid in the plane such that
and each edge is contained in a regular square (resp. hexagon) and every interior face is surrounded by a regular square (resp. hexagon).
A polyomino graph  is  a \emph{straight chain} if it is 2-connected and each vertex is contained in at most two squares.
A hexagonal system $H$ is called \emph{all-kink catahex}
if each vertex is contained in at most two hexagons and  there is no hexagon
such that its intersections with two other hexagons are two parallel edges (see Fig. \ref{example}).

For  a straight chain polyomino
(resp. an all-kink catahex) $G$,
we can construct a perfect matching of $G$ such that each square (resp. hexagon) is $M$-alternating. 
By Theorems \ref{antiforcingset} and \ref{maxantiforcing},
we have  $af(G,M)=Af(G)=r(G)$.
Conversely, suppose that a hexagonal system (resp. polyomino) $G$ has a perfect matching $M$ such that $af(G,M)=Af(G)=r(G)$.  By Lemma \ref{general} we have that $G$ is elementary since it has no cut edges.  Lemma \ref{cha} implies that
$G$ has an $M$-alternating hexagon (resp. square) which has at most one adjacent hexagon (resp. square) and the common edge belongs to $M$. So by the inductive way we can confirm the necessities of  the following results.

\begin{corollary}\label{hex}{\bf .}\cite{Zh2}
Let $G$ be a  hexagonal system with $n$ hexagons.
Then $Af(G)=n$ if and only if $G$ is  an all-kink catahex.
\end{corollary}

\begin{corollary}\label{poly}{\bf .}
Let $G$ be a polyomino graph  with $n$ squares.
Then $Af(G)=n$ if and only if $G$ is a straight chain.
\end{corollary}

\section{Anti-forcing spectra of even polygonal chains}

Let $G$ be a connected plane graph. The boundary of $G$ means the boundary of the exterior face of $G$. $G$ is called \emph{outerplane} if all the vertices of $G$ are on the boundary. A \emph{face cycle} of $G$ means the boundary of an interior face if it is a cycle.
A 2-connected plane bipartite graph $G$ is called an \emph{even polygonal chain}
if  $G$ is outerplane  and  each face cycle
is adjacent at most two face cycles.
As special cases, some molecular graphs are even polygonal chains,
such as hexagonal chains and phenylene chains. By Theorem \ref{face}, an even polygonal chain is elementary.
By Theorem 2.10 in \cite{Zhang},
each interior face of an even polygonal chain is resonant.


An edge of an even polygonal chain $G$ is called \emph{boundary edge}
if it is on the boundary of  $G$. An edge of $G$ is called \emph{shared}
if it is the common edge of two face cycles. A face cycle of $G$ is called {\em terminal} if it has only one shared edge of $G$. A face cycle $s$ of $G$ is called a \emph{kink} if it has two shared edges that can be contained simultaneously
in a perfect matching of $s$, that is, the two  edges  go clockwise along $s$ from white (resp. black) ends to
black (resp. white) ends. If all the face cycles of $G$ except for the terminal are kinks,
then we say $G$  \emph{all-kink}. For example, a straight chain polyomino is all-kink.
$G$ is called a \emph{linear chain}
if it has no kinks.
A subchain of $G$ is called a \emph{maximal
linear chain} if it can not be contained in a linear
subchain with more face cycles. 

\begin{lemma}\label{resonantface}\cite{Zhang}{\bf .}
Let $G$ be a plane elementary bipartite graph with a perfect matching
$M$, $C$ an $M$-alternating cycle. Then there exists an $M$-resonant
face in the interior of $C$.
\end{lemma}

\begin{lemma}\label{z}\cite{Zhang}{\bf .}
Let $G$ be a plane elementary bipartite graph.
Then $Z(G)$ is connected.
\end{lemma}

Let $G$ be a plane bipartite graph with a perfect matching $M$. Given a compatible $M$-alternating set $\mathcal{A}$, two cycles $C_1$ and $C_2$ of $\mathcal{A}$ are \emph{crossing} if they share an edge $e$ in $M$ and the four edges adjacent to $e$
alternate in $C_1$ and $C_2$ (i.e., $C_1$ enters into $C_2$ from one side and leaves for the other side via $e$).
$\mathcal{A}$ \emph{is non-crossing}  if any two cycles in $\mathcal{A}$ are not crossing.
For hexagonal systems,
Lei et al.  \cite{Zh2}
proved that any compatible $M$-alternating set can be changed to
  a non-crossing compatible $M$-alternating set with the same cardinality.
This result can be generalized to plane bipartite graphs in the same way as follows.

\begin{lemma}\label{nonc}{\bf .} If $G$ has
a compatible $M$-alternating set $\mathcal{A}$, then
 $G$  admits a non-crossing compatible
$M$-alternating set with size $|\mathcal{A}|$.
\end{lemma}

\subsection{Continuity}
In this subsection we show that the anti-forcing spectrum of
any even polygonal chain is an integer interval.

Let $G$ be an even polygonal chain with a perfect matching $M$. For
an $M$-alternating cycle $C$ in $G$, let $f(C)$ be the number of faces in the interior of $C$.
By Theorem \ref{compatibleset} and Lemma \ref{nonc},
we can choose a maximum non-crossing compatible $M$-alternating
set $\mathcal{A}$ such that $|\mathcal{A}|=af(G,M)$ and
$f(\mathcal{A})=\sum_{C\in \mathcal{A}}f(C)$
is as small as possible.
By using these notations, we have following lemmas.

\begin{lemma}\label{minindex}{\bf .}
$\mathcal{A}$ contains all $M$-alternating face cycles in $G$,
and any two non-face cycles in $\mathcal{A}$ are inner
disjoint and have at most one common edge in $M$.
\end{lemma}

\begin{proof}
Suppose that $G$ has  an $M$-alternating face cycle $s\notin \mathcal{A}$.
Then there exists a cycle $C\in\mathcal{A}$
such that $C$ and $s$ are not compatible. That is,
$C$ and $s$ have a common edge not in $M$. So $s$ must be in the interior of $C$.
We claim that $(\mathcal{A}\setminus\{C\})\cup \{s\}$
is a compatible $M$-alternating set.
Otherwise there is a cycle $C^\prime\in \mathcal{A}\setminus\{C\}$
which is not compatible with $s$.
So $s$ is in the interior of $C^\prime$. This
 implies that $C$ and $C^\prime$ are
either incompatible or crossing, a contradiction.
 Since $f(C)>1$, $f((\mathcal{A}\setminus\{C\})\cup \{s\})<f(\mathcal{A})$, a contradiction.

Suppose $\mathcal{A}$ has   two non-face cycles $C_1$ and $C_2$ such that
 $C_1$ is contained in the interior of $C_2$.
By Lemma \ref{resonantface},
there exists  an $M$-alternating face cycle $s_1$ in the interior
of $C_1$. By the above proof, $s_1\in \mathcal A$.
Since $G$ is a chain,
$s$ and $C_1$ must have two common edges, which belong to $M$ and also to $C_2$.  This implies that $C_1$ and $C_2$ have a common edge not in $M$, a contradiction.
Hence any two cycles in $\mathcal{A}$ are inner disjoint and have at most one common edge in $M$.
\end{proof}

\begin{lemma}\label{unique}{\bf .}
 $G$ has a unique maximum
non-crossing  compatible $M$-alternating set $\mathcal{A}_M$
such that $f(\mathcal{A}_M)$ is minimum.
\end{lemma}

\begin{proof}
To the contrary,
suppose  there is another maximum
non-crossing  compatible $M$-alternating set  $\mathcal{A}$
such that $f(\mathcal{A})=f(\mathcal{A}_M)$.
By Lemma \ref{minindex},
all $M$-alternating face cycles belong to
$\mathcal{A}\cap\mathcal{A}_M$.
Since $\mathcal{A}$ and $\mathcal{A}_M$
are different,
there exists a pair of non-face cycles $C\in \mathcal{A}$
and $C^\prime\in \mathcal{A}_M$ that are incompatible.
Since $G$ is outerplane,
$C$ and $C^\prime$ must have a common region in
their interiors.
Let $Q$ be the boundary of the common region.
Then $Q$ is an $M$-alternating cycle.
Note that $f(Q)$ is less than at least one of $f(C)$ and $f(C^\prime)$,
say  $f(C)$.
By Lemma \ref{minindex},
any cycle of $\mathcal{A}\setminus\{C\}$
is non-crossing and compatible with $Q$.
Therefore
$(\mathcal{A}\setminus\{C\})\cup\{Q\}$
is a maximum non-crossing compatible $M$-alternating set with $f((\mathcal{A}\setminus\{C\})\cup\{Q\})<f(\mathcal{A})$, a contradiction.
\end{proof}

\begin{lemma}\label{gap}{\bf .}
For two perfect matchings  $M$ and $M^\prime$
of  $G$ such that $s=M\triangle M^\prime$ is a face cycle of $G$,
 $|af(G,M)-af(G,M^\prime)|\leq2$.
If $af(G,M)-af(G,M^\prime)=2$, then (i) $s$ is a  kink that is not 4-cycle, (ii) $\mathcal{A}_{M}\setminus\{s\}$ contains
exactly two cycles $C_e$ and $C_f$  passing through
the two shared edges $e$ and $f$ of $s$ respectively, and (iii) $\mathcal{A}_{M^\prime}=\mathcal{A}_M\setminus\{C_e,C_f\}$.
\end{lemma}

\begin{proof}
By Lemma \ref{minindex}, $s\in \mathcal{A}_{M}\cap \mathcal{A}_{M^\prime}$.
There are two cases to be considered.

\textbf{Case 1.}
$s$ is not a kink of $G$.
Since $s$ is both $M$ and $M'$-alternating,
$s$ has at most one  shared edge  in
 $M$.
Then there is at most one possible
cycle $C\in\mathcal{A}_M\setminus\{s\}$
which passes through some edge of $s$, and
$\mathcal{A}_M\setminus\{C\}$ also is
a compatible $M^\prime$-alternating set
since $M$ and $M^\prime$  differ only on $s$.
So $|\mathcal{A}_M|-1\leq|\mathcal{A}_{M^\prime}|$.
Similarly, $|\mathcal{A}_{M^\prime}|-1\leq|\mathcal{A}_{M}|$.
Therefore  $||\mathcal{A}_M|-|\mathcal{A}_{M^\prime}||\leq1$.

\textbf{Case 2.}
$s$ is a kink of $G$.
Let $e$ and $f$ be the two shared edges of $s$. Then $\{e,f\}\subset M$ or $M^\prime$,
say $M$. There is at most one cycle $C\in\mathcal{A}_{M^\prime}\setminus\{s\}$
such that $s$ is in the interior of $C$, so $\mathcal{A}_{M^\prime}\setminus\{C\}$
is a compatible $M$-alternating set. Hence $|\mathcal{A}_{M^\prime}|-1\leq |\mathcal{A}_{M}|$.

On the other hand,  there is at most two possible cycles
$C_e$ and $C_f$ in  $\mathcal{A}_M\setminus\{s\}$
that pass $e$ and $f$ respectively. Then
$\mathcal{A}_M\setminus\{C_e,C_f\}$  is a
compatible $M^\prime$-alternating set.
So $|\mathcal{A}_M|-2\leq |\mathcal{A}_{M^\prime}|$. So the first part of the lemma holds.

From now on suppose  $|\mathcal{A}_M|-|\mathcal{A}_{M^\prime}|= 2$. Then members $C_e$ and $C_f$ in  $\mathcal{A}_M\setminus\{s\}$ must exist and $\mathcal{A}':=\mathcal{A}_M\setminus\{C_e,C_f\}$ is a maximum compatible $M^\prime$-alternating set. This implies that $\mathcal{A}'$ has no a cycle containing  $s$ in its interior. We have that $s$ is not 4-cycle. Otherwise, $C：=(C_e\cup C_f)\triangle s$ is an $M'$-alternating cycle that is compatible with each cycle in $\mathcal{A}'$, contradicting that $\mathcal{A}'$ is maximum. So statements  (i) and (ii) hold.

The remaining is to prove that $\mathcal{A}_{M'}=\mathcal{A}'$.
Suppose $\mathcal{A}_{M^\prime}\neq\mathcal{A}_M\setminus\{C_e,C_f\}$.
By Lemma \ref{unique},
$f(\mathcal{A}_{M^\prime})<f(\mathcal{A}_M\setminus\{C_e,C_f\})$.
So $\mathcal{A}_{M^\prime}$
contains a non-face cycle $C^\prime$
such that $C^\prime$ is incompatible with
$C_e$ or $C_f$ with respect to $M$, say $C_e$. Otherwise $\mathcal{A}_{M^\prime}\cup\{C_e,C_f\}$
can be a maximum non-crossing  compatible $M$-alternating set
with $f(\mathcal{A}_{M^\prime}\cup\{C_e,C_f\})<f(\mathcal{A}_M)$,
a contradiction.
Since $G$ is a chain,
the interiors of $C^\prime$ and $C_e$ must have a common region.
Let $Q^\prime$ be the boundary of the common area.
Then $Q^\prime$ is an $M$-alternating cycle.
Note that $Q^\prime$ is compatible with
any $M$-alternating face cycle,
and $f(Q^\prime)<f(C_e)$.
By Lemma \ref{minindex},
$Q^\prime$ is compatible with each cycle of
$\mathcal{A}_{M}\setminus\{C_e\}$.
Hence $(\mathcal{A}_{M}\setminus\{C_e\})\cup\{Q^\prime\}$
can be a maximum non-crossing compatible $M$-alternating
set with $f((\mathcal{A}_{M}\setminus\{C_e\})\cup\{Q^\prime\})<f(\mathcal{A}_{M})$,
a contradiction.
\end{proof}

\begin{figure}[h]
  \centering
    \includegraphics[width=60mm]{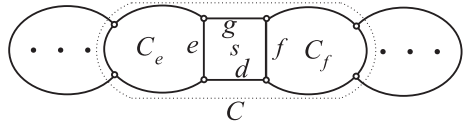}
       \caption{Illustration for the proof of  Lemma \ref{gap1},
       $s$ is a square.}
  \label{square}
\end{figure}
\begin{lemma}\label{gap1}{\bf .}
If $\mathcal{A}_M$ contains a non-face cycle $C$,
then there is an $M$-alternating 4-cycle $s$ in the interior
of $C$
such that  $af(G,M')=af(G,M)+1$, where  $M^\prime=M\triangle s$.
\end{lemma}
\begin{proof}
By Lemma \ref{resonantface},
there is an $M$-alternating face cycle $s$
in the interior of $C$.
By Lemma \ref{minindex}, $s\in \mathcal{A}_M$.
So $s$ and $C$ are compatible.
Since $G$ is a chain,
$s$ must be a 4-cycle.
By Lemma \ref{gap},
$||\mathcal{A}_{M^\prime}|-|\mathcal{A}_M||\leq1$.
Let $g$ and $d$ (resp. $e$
and $f$) be the  boundary (resp. shared) edges of $G$ in $s$.
Then $g,d \in M$ and $e,f \in M^\prime$.
See Fig. \ref{square},
$C-g-d+e+f$ is the disjoint union of two
cycles $C_e$ and $C_f$ which pass through $e$ and $f$ respectively.
Since $M$ and $M^\prime$ differ only
on $s$, $C_e$ and $C_f$ both are $M^\prime$-alternating,
and $(\mathcal{A}_M\setminus\{C\})\cup\{C_e,C_f\}$
is a compatible $M^\prime$-alternating set.
Hence $|\mathcal{A}_M|+1\le
 |\mathcal{A}_{M^\prime}|$,
which implies $|\mathcal{A}_M|+1=|\mathcal{A}_{M^\prime}|$.
By Theorem \ref{compatibleset},
$af(G,M^\prime)=af(G,M)+1$.
\end{proof}

According to Lemma \ref{gap1}, the following result is immediate.

\begin{corollary}\label{max}{\bf .}
For a perfect matching $M$ of  $G$ with $af(G,M)=Af(G)$,
 $af(G,M)$ equals the number of $M$-resonant faces.
\end{corollary}

\begin{figure}[h]
  \centering
    \includegraphics[width=90mm]{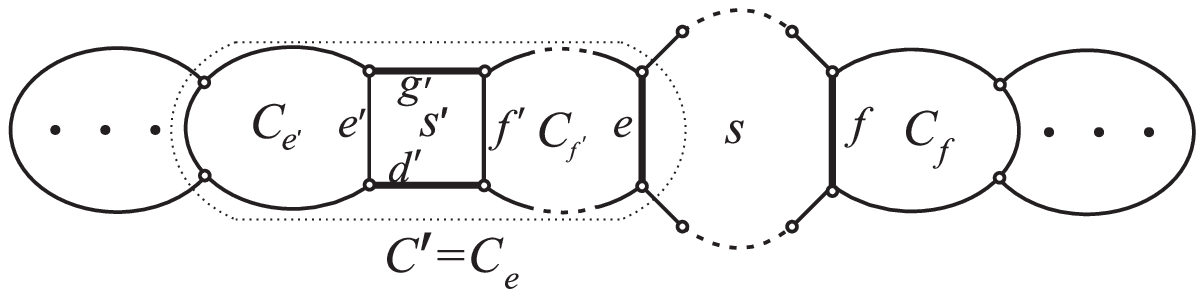}
       \caption{Illustration for the proof of Lemma \ref{gap2},
       $s^\prime$ is a square.}
  \label{square1}
\end{figure}

\begin{lemma}\label{gap2}{\bf .}
Let $M$ and $M^\prime$ be perfect matchings
of  $G$
such that $s=M\triangle M^\prime$ a face cycle in $G$.
If $af(G,M)-af(G,M^\prime)=2$,
then there is a perfect matching $M^{\prime\prime}$
in $G$ such that $af(G,M^{\prime\prime})=af(G,M^\prime)+1$.
\end{lemma}

\begin{proof}
By Lemma \ref{gap},
$s$ is a kink of length at least 6 and has two shared edges $e$ and $f$
in $M$,  $\mathcal{A}_M\setminus\{s\}$
contains exactly two cycles $C_e$ and $C_f$  passing through
$e$ and $f$ respectively, and
$\mathcal{A}_{M^\prime}=\mathcal{A}_M\setminus\{C_e,C_f\}$.

Suppose that $\mathcal{A}_M$ contains a non-face cycle $C^\prime$. By Lemma \ref{gap1},
there is an $M$-alternating face cycle $s^\prime$ of length 4 in the
interior of $C^\prime$. Note that $s'$ is disjoint with $s$ and  thus $M'$-alternating.
By Lemma \ref{minindex},
$s,s^\prime\in\mathcal{A}_M\cap\mathcal{A}_{M^\prime}$.
If $C^\prime$ also belongs to $\mathcal{A}_{M^\prime}$,
by Lemma  \ref{gap1},
then $af(G,M^{\prime\prime})=af(G,M^{\prime})+1$, where $M^{\prime\prime}=M^{\prime}\triangle s^{\prime}$.
For the case $C^\prime\notin \mathcal{A}_{M^\prime}$,
$C^\prime$ is either $C_e$ or $C_f$,
say $C^\prime=C_e$.
Let $g^\prime$ and $d^\prime$
(resp. $e^\prime$ and $f^\prime$)
be the two boundary (resp. shared) edges  of $s^\prime$.
Then $C^{\prime}-g^{\prime}-d^{\prime}+e^{\prime}+f^{\prime}$
is the disjoint union of two cycles $C_{e^\prime}$ and $C_{f^\prime}$
such that  $e^\prime\in C_{e'}$ and $ f^\prime,e\in C_{f'}$ （see Fig. \ref{square1}).
$C_{e^\prime}$ is  $M^{\prime\prime}$-alternating,
so $\mathcal{A}_{M^\prime}\cup\{C_{e^\prime}\}$ is a compatible
$M^{\prime\prime}$-alternating set.
Hence $|\mathcal{A}_{M^\prime}|+1
\leq |\mathcal{A}_{M^{\prime\prime}}|$.
On the other hand,
by Lemma \ref{gap}, $af(G,M^{\prime\prime})\le af(G,M^\prime)+1$.
By Theorem \ref{compatibleset},
$af(G,M^{\prime\prime})=af(G,M^\prime)+1$.

From now on suppose all  members of $\mathcal{A}_M$
are face cycles.
So $C_e$ and $C_f$ both are face cycles and
the number of $M$-resonant faces of $G$ is $|\mathcal{A}_M|$.
Let $G_M$ be the graph formed by all $M$-alternating  face cycles of $G$.
Then each component of $G_M$
is an all-kink subchain.
A component of $G_M$ is called a \emph{single component}
if it is a single face cycle, is \emph{non-single component} otherwise.
It is obvious that $G_M$ has a component containing three face cycles  $C_e,s$ and $C_f$.
So $G_M$ has at least one non-single component.

For convenience, we define an orientation of chain $G$ from left to right whenever we go along it from $C_e,s$ to $C_f$.
Let $L_0$ be the leftmost non-single component
of $G_M$ in $G$. Let $s_0$ be the left terminal face cycle in $L_0$,
and $s_0^-$ and $s_0^+$ the left and right neighboring
face cycles of $s_0$ in $G$ respectively. It is possible that $s_0^+=s$.
Let $M_1=M\triangle s_0$. By the first part of Lemma \ref{minindex},  $\mathcal{A}_M\setminus\{s_0^+\}\subseteq \mathcal{A}_{M_1}$. So  $|\mathcal{A}_M|-1\leq|\mathcal{A}_{M_1}|$.

If $s_0$ is the left terminal face cycle in $G$
(i.e., $s_0^-$ does not exist),
then  $|\mathcal{A}_{M_1}|=|\mathcal{A}_{M}|-1$.
Letting $M^{\prime\prime}=M_1$,  $af(G,M^{\prime\prime})=af(G,M^\prime)+1$.
So  suppose that $s_0^-$ exists.
Let $e_0$
and $f_0$
be the left and right shared edges of $s_0$ respectively.
Then $f_0\in M$.

Let $e_0\in M$.  If $\mathcal{A}_{M_1}$  contain  a cycle $C_0$
such that $s_0$ is in its interior, then $s_0$ must be a 4-cycle. Note that $M=M_1\triangle s_0$.
By Lemma \ref{gap1}, $af(G,M)=af(G,M_1)+1$.
Otherwise  $\mathcal{A}_M\setminus\{s_0^+\}= \mathcal{A}_{M_1}$.
Then $af(G,M_1)=af(G,M)-1=af(G,M^\prime)+1$.

So suppose that  $e_0\notin M$. Then $e_0\in M_1$ and $f_0\notin M_1$. If $\mathcal{A}_M\setminus\{s_0^+\}= \mathcal{A}_{M_1}$, then $|\mathcal{A}_{M_1}|=|\mathcal{A}_M|-1$ and the result holds. Otherwise, $ \mathcal{A}_{M_1}$ must contain just one $M_1$-alternating cycle $D_1$ which is compatible with each cycle of $\mathcal{A}_M\setminus\{s_0^+\}$. So $D_1$ passes $e_0$ and contains $s_0^-$ in its interior. So
$|\mathcal{A}_{M_1}|=|\mathcal{A}_M|$.
If $D_1$ is a non-face cycle,
by Lemma \ref{gap1},
there is a 4-cycle $s_0^\prime$ in the interior of  $D_1$
such that $M_1\triangle s_0^\prime$ is a perfect matching of $G$ with
$|\mathcal{A}_{M_1\triangle s_0^\prime}|=|\mathcal{A}_{M_1}|+1$.
It implies that  $|\mathcal{A}_{M^\prime\triangle s_0^\prime}|=|\mathcal{A}_{M^\prime}|+1$.
Let $M^{\prime\prime}=M^\prime\triangle s_0^\prime$.
Then $af(G,M^{\prime\prime})=af(G,M^\prime)+1$.
Now suppose $D_1$ is an $M_1$-alternating face cycle.
Then $D_1=s_0^-$, and each cycle of $\mathcal{A}_{M_1}$ is a face cycle.
It clear that there is a non-single component in
$G_{M_1}$ since $s_0$ and $s_0^-$ are adjacent $M_1$-alternating face cycles.
Let $L_1$ be the most left non-single component of $G_{M_1}$ in $G$,
and $s_1$ the left terminal face cycle of $L_1$.
Then  $s_1$ lies on the left side of $s_0$ in $G$.
Repeating the above procedure, finally
we can  find a perfect matching $M^{\prime\prime}$ in $G$
such that $af(G,M^{\prime\prime})=af(G,M^\prime)+1$ since
$G$ has a finite number of face cycles.
\end{proof}

\begin{theorem}\label{continuous}{\bf .}
The anti-forcing spectrum of an even polygonal chain $G$ is continuous.
\end{theorem}

\begin{proof}
Let $M$ and $M^\prime$ be two perfect matchings in $G$
such that  $af(G,M)=af(G)$ and $af(G,M^\prime)=Af(G)$.
By Lemma \ref{z},
there is a path $M_1(=M)M_2\cdots M_n(=M^\prime)$ in $Z(G)$.
By Lemma \ref{gap}, $|af(G,M_i)-af(G,M_{i+1})|\leq2$ for all $1\leq i<n$.
If $|af(G,M_i)-af(G,M_{i+1})|=2$,
by Lemma \ref{gap2},
there is a perfect matching $M^{\prime\prime}$ in $G$
such that  $|af(G,M_i)-af(G,M^{\prime\prime})|=1$ and $|af(G,M^{\prime\prime})-af(G,M_{i+1})|=1$.
Therefore, there is no gap in Spec$_{af}(G)$.
\end{proof}

%

\subsection{Minimum anti-forcing number}

\begin{figure}[h]
  \centering
    \includegraphics[width=120mm]{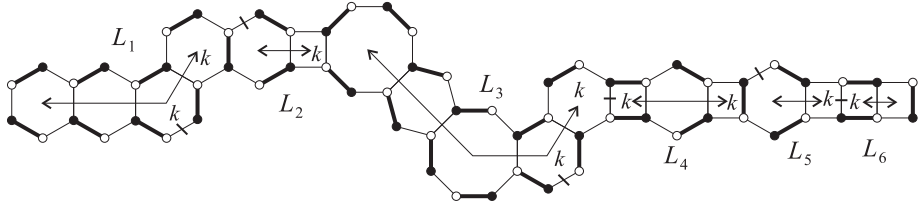}
       \caption{Segment decomposition of $G$, kinks are signed by $k$, anti-forcing edge of each segment is marked by a bar.}
  \label{decomposition}
\end{figure}

Let $G$ be an even polygonal chain.
Label the face cycles of $G$ as $s_1,s_2,\ldots,s_n$ from left to right. If $G$ has  no kink, then itself is a segment.
Let $s_i$ be the first  kink of $G$.
If $s_i$ is a 4-cycle, then the subchain $s_1+s_2+\cdots+s_i$
is called the first segment of $G$.
Otherwise, the subchain $s_1+s_2+\cdots+s_j$ or $s_1+s_2+\cdots+s_n$
is called the first segment of $G$, denoted by $L_1$, according as $G$ has the second kink $s_j$ or no.
Let $G-L_1$ be the subchain
obtained by removing the face cycles of $L_1$ from $G$. Precisely, $G-L_1$ consists of all face cycles of $G$ but not in $L_1$.
If $L_1\not= G$, similarly let $L_2$  be the first segment of $G-L_1$ as the second segment of $G$.
Repeating this procedure,
$G$ can be decomposed into a sequence of segments
($L_1,L_2,\ldots,L_t$), $t\ge 1$,
called the \emph{segment decomposition} of $G$ (see Fig. \ref{decomposition}).
Obviously, each segment $L_i$ has  an anti-forcing edge by Theorem \ref{antiforcingedge}
(see Fig. \ref{decomposition}). Since $G$ is elementary, by Theorem \ref{eardecomposition} or \ref{face} $G$  has at least two perfect matchings, $af(G)\ge 1$.

%
%

\begin{theorem}\label{min}{\bf .}
Let ($L_1,L_2,\ldots,L_t$) be the segment decomposition
of an even polygonal chain $G$.
Then $af(G)=t$.
\end{theorem}
\begin{proof}
We proceed by induction on $t$.
For the case $t=1$,
$G$ itself is a segment,
so $G$ has an anti-forcing edge,
and  $af(G)=1$.
Suppose $t\geq 2$.
The anti-forcing edges of
all segments $L_i$ ($i=1,2,\ldots,t$)
form an anti-forcing set
of a perfect matching of $G$.
So $af(G)\leq t$.
We only need to prove $ af(G)\ge t$. Let $s_{i_0}$ be the right terminal face cycle of $L_1$.

Let $M$ be a perfect matching of $G$ with $af(G,M)=af(G)$.
Note that ($L_2,L_3,\ldots,L_t$) is the segment decomposition of $G-L_1$.
By the induction hypothesis, we have that $af(G-L_1)=t-1$.
There are two cases to be considered.

\textbf{Case 1.}
The restriction of $M$ to $G-L_1$ is a perfect matching
of $G-L_1$. Since $s_{i_0}$ is a kink of $G$,  the restriction of $M$ to $L_1-s_{i_0}$ is a perfect matching
of $L_1-s_{i_0}$. Let $S$ be any minimum anti-forcing set  of $M$. By Theorem \ref{antiforcingset}, $S\cap E(G-L_1)$ and $S\cap E(L_1-s_{i_0})$ are anti-forcing sets of $G-L_1$ and $L_1-s_{i_0}$ respectively.
Since $af(G-L_1)=t-1$ and $af(L_1-s_{i_0})\geq 1$,  $S$ contains at least $t-1$ edges in $G-L_1$ and at least one edge of $L_1-s_{i_0}$.
So $t\leq|S|$, i.e.,  $t\leq af(G)$.

\textbf{Case 2.}
The restriction of $M$ to $G-L_1$ is not a perfect matching
of $G-L_1$.
Let $e_1$ (resp. $f_1$) be the shared edge between
$L_1-s_{i_0}$ (resp. $G-L_1$) and $s_{i_0}$.
Then $e_1,f_1\notin M$.
Since $s_{i_0}$ is a kink of $G$,
$s_{i_0}$ is $M$-alternating.
Let $M_1=M\triangle s_{i_0}$.
Then $M_1$ is a perfect matching of $G$
and $e_1,f_1\in M_1$.
By Lemma \ref{gap},
$af(G,M_1)\leq af(G,M)+2$.
Note that the restriction of $M_1$ to $G-L_1$ is a perfect matching
of $G-L_1$.
Also,
any minimum anti-forcing set $S_1$ of $M_1$
contains at least $t-1$ edges of $G-L_1$.
On the other hand,
the restriction of $M_1$ to $L_1-s_{i_0}$ also is a perfect matching
of $L_1-s_{i_0}$. $L_1-s_{i_0}$ contains an $M_1$-alternating cycle $C_1$
which is compatible with $s_{i_0}$.
By Theorem \ref{antiforcingset},
$S_1$ must contain at least two edges of $L_1$.
If $af(G,M_1)\leq af(G,M)+1$,
then $t+1\leq|S_1|\leq af(G,M)+1$,
i.e., $t\leq af(G)$.

Suppose that  $af(G,M_1)=af(G,M)+2$. By Lemma \ref{gap},
$s_{i_0}$ is not a 4-cycle, and
$\mathcal{A}_{M_1}\setminus\{s_{i_0}\}$
includes a cycle $C_{e_1}$  in $L_1-s_{i_0}$ passing through $e_1$.
If $C_{e_1}$ is not a face cycle,
by Lemma \ref{resonantface},
then there is an $M_1$-alternating face cycle $s^\prime$
in the interior of $C_{e_1}$.
By Lemma \ref{minindex},
$s'$, $s_{i_0}$ and $C_{e_1}$ all belong to $\mathcal{A}_{M_1}$,
which implies that $S_1$ includes at least three edges of $L_1$.
Consequently, $t+2\leq|S_1|= af(G,M)+2$,
i.e., $t\leq af(G)$.
So we may assume that $C_{e_1}=s_{i_0-1}$ is a face cycle.
Since $s_{i_0}$ is not a 4-cycle,
$L_1$ contains another kink $s_{i}$ of $G$, $2\le i\le i_0-1$. Let $L_1^\prime$ be the subchain consisting of face cycles $s_1,s_2,\ldots, s_{i-1}$.
It follows that the restriction of $M_1$ to $L_1^\prime$
is a perfect matching of $L_1^\prime$.
Also there is an $M_1$-alternating face cycle
$C_1^\prime$ in $L_1^\prime$.
Hence $s_{i_0}$, $s_{i_0-1}$ and $C_1^\prime$ are three compatible
$M_1$-alternating cycles in $L_1$.
Like the above, we have  $t\leq af(G)$.
\end{proof}

Actually,  Theorem \ref{min} gives a  linear algorithm
to compute the anti-forcing number of any even polygonal chain.
For example, the anti-forcing number of the even polygonal chain
in Fig. \ref{decomposition} is 6. As special cases,
the anti-forcing numbers of hexagonal chains and
phenylene chains have been computed in
\cite{Deng1} and \cite{ZhangQQ}, respectively.

Note that each segment of an all-kink
even polygonal chain has at most three face cycles. For a real number $x$, let $\lceil x\rceil$ be the smallest
integer not less than $x$.
By Theorem \ref{min}, we have the following result.

\begin{corollary}\label{errata-}{\bf .}
Let $G$ be an all-kink even polygonal chain without 4-cycles  and with $n$ face cycles.
Then $af(G)=\lceil\frac{n}{3}\rceil$.

\end{corollary}

\begin{corollary}\label{errata}{\bf .}
Let $G$ be a straight chain of $n$ squares.
Then $af(G)=\lceil\frac{n}{2}\rceil$.
\end{corollary}

For any hexagonal chain $H$, Theorem 5 in \cite{VT2} stated that
$af(H)=\lceil\frac{k(H)}{2}\rceil$, where
$k(H)$ is the number of maximal linear chains of $H$. In fact this is not correct in general. For example,
for an all-kink hexagonal chain $L$ with $n\geq 2$ hexagons, from Corollary \ref{errata-} we have $af(L)=\lceil\frac{n}{3}\rceil$. But $k(L)=n-1$, and $\lceil\frac{n-1}{2}\rceil>\lceil\frac{n}{3}\rceil$
for all $n\geq8$.


\subsection{Maximum anti-forcing number}

\begin{figure}[h]
  \centering
    \includegraphics[width=80mm]{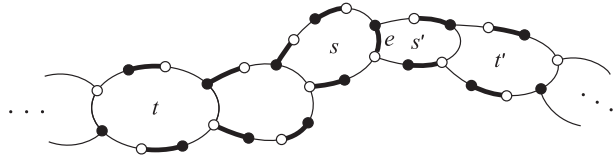}
       \caption{A maximal linear chain with terminals $t$ and $t'$.}
  \label{attwo}
\end{figure}

For an  even polygonal chain $G$, from Lemma \ref{cha} we can see that $Af(G)=r(G)$ if and only if $G$ is all-kink. Here the cyclomatic number $r(G)$ equals the number of interior faces of $G$.  Our approach to compute  $Af(G)$ is to decompose $G$ into some all-kink subchains that contain face cycles as many as possible.

\begin{lemma}\label{two}{\bf .}
Let $M$ be a perfect matching of  $G$ and
$B$ a maximal linear chain of $G$.
Then $B$ contains at most two $M$-alternating face cycles.
Further, if $B$ contains exactly two $M$-alternating face cycles,
then they are neighboring.
\end{lemma}
\begin{proof}
Suppose $B$ has an $M$-alternating face cycle $s$.
Let $e$ be the right shared edge of $s$ (see Fig. \ref{attwo}). If $e\notin M$, then  no right face cycles of $s$ in $B$ are $M$-alternating. Suppose $e\in M$ and $s$ has a right neighboring face cycle $s'$. Then $s'$  may be $M$-alternating and the other right faces (if there exist) of $s$ in $B$ are not $M$-alternating (in fact, if $s'$ is non-kink, then $s'$ must be $M$-alternating). On the left side of $s$, the corresponding facts also hold.  So $B$ has three $M$-alternating face cycles only if both shared edges of $s$ belong to $M$. In this case, $s$ itself is a kink, i.e.,  a terminal face cycle of $B$. Hence the lemma holds.
\end{proof}

Let $G$ be a graph with a perfect matching $M$,
and $G^\prime$ is a subgraph of $G$ such that the restriction of $M$ to $G^\prime$ is a perfect matching of $G^\prime$.
Then $G-V(G')$ either has a perfect matching or is empty. If $G-V(G')$ has pendant edges (with an end  of degree one), then delete the ends of those pendant edges  and their incident edges from $G-V(G')$.
Repeating this process,
until we obtain a graph  without pendant edges,
denoted by $G\ominus G^\prime$.
Obviously, all these pendant edges happen in this process are contained in $M$.

\begin{figure}[h]
  \centering
    \includegraphics[width=120mm]{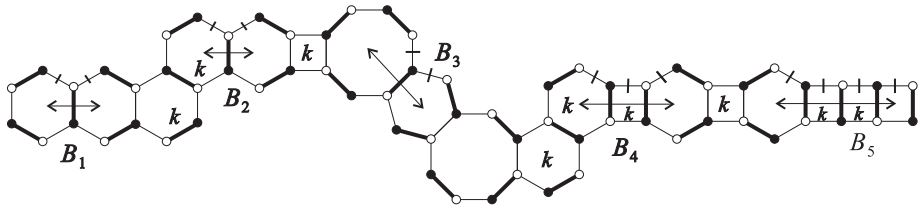}
       \caption{An all-kink decomposition ($B_1,B_2,B_3,B_4,B_5$) of $G$, kinks are marked by $k$ and bold edges form a perfect matching $F$.}
  \label{sequence}
\end{figure}

Let $B_1=s_1+s_2+\cdots+s_i$ ($1\leq i\leq n$)
be the first  maximal all-kink subchain of $G_1$ from the left terminal $s_1$.
That is, $B_1$ itself is an all-kink chain,
but $B_1+s_{i+1}$ is not.
If  $G_2:=G\ominus B_1\not=\emptyset$, then subchain $G_2$ has the first  maximal all-kink subchain $B_2$.
Let $G_3=G_2\ominus B_2$.
Keeping on this procedure, we  finally obtain an even polygonal chain $G_m$
such that $G_m\ominus B_m=\emptyset$,
where $B_m$ is the first maximal all-kink subchain of $G_m$. Then ($B_1,B_2,\ldots,B_m$)
is called an \emph{all-kink decomposition} of $G$ (see Fig. \ref{sequence}). Note that the last face cycle of each $B_t$ is not kink of $G$.
By using these notations,
we give the following result.

\begin{theorem}\label{maxn}{\bf .}
$Af(G)=\sum_{t=1}^{m}r(B_{t})$.
\end{theorem}

\begin{proof}
Let $k$ be the number of kinks of $G$.
We proceed by induction on $k$.
We first consider the case $k=0$. Then $G$ is a linear chain. By Corollary \ref{max} and Lemma \ref{two},
$Af(G)\leq 2$.
If $G$ is a single cycle, the result is trivial.
So assume  $r(G)\geq 2$. Then $r(B_1)=2$  and $G\ominus B_1=\emptyset$.
$G$ has a perfect matching $M$  such
that first two  face cycles are $M$-alternating. Hence  $2\leq af(G,M)\leq Af(G)$.
So $Af(G)=r(B_1)=2$.

Suppose  $k\geq 1$. $G$ has a perfect matching $F$ such that all shared edges of each $B_t$ ($1\leq t\leq m$) belong to $M$.
So all face cycles  of each $B_t$ are $F$-alternating.
So those $F$-alternating face cycles form
a compatible $F$-alternating set of $G$.
By Theorem \ref{compatibleset},
$\sum_{t=1}^{m}r(B_{t})\leq af(G,F)\leq Af(G)$.
In the following,
we want to prove $\sum_{t=1}^{m}r(B_{t})\geq Af(G)$. To this end, we choose a perfect matching $M$ of $G$ such that $af(G,M)=Af(G)$.

If $G\ominus B _1=\emptyset$, then $G-B_1$ is empty or a linear chain.
Let $s_{i-1}$ and $s_i$ be the last two face cycles of $B_1$.
Since $k\geq1$, $i\geq3$ and  $s_{i-1}$ is the last kink of $G$.
So $J=s_{i-1}+s_i+(G-B_1)$ is a maximal linear chain of $G$.
By Lemma \ref{two},
there are at most two $M$-resonant faces in $J$, and at most $(r(B_1)-2)$
 $M$-resonant faces in $B_1-s_{i-1}-s_i$.
By Corollary \ref{max}, $Af(G)\leq r(B_1)$.

Now suppose $G\ominus B _1\neq\emptyset$.
Note that the last face cycle $s_i$ of $B_1$
is not a kink of $G$.
So $G$ must have a kink after $s_i$.
Let $s_j$ ($i<j<n$) be the first kink of $G$ in $G-B_1$.
Then $G_2=G\ominus B_1=G-s_1-s_2-\cdots-s_j$.
Note that ($B_2,B_3,\ldots,B_m$) is an all-kink decomposition of $G_2$. Since $G_2$ has less kinks than $G$,
by the induction hypothesis,
$Af(G_2)=\sum_{t=2}^{m}r(B_t)$.

Let $H_1=B_1+s_{i+1}+s_{i+2}+\cdots+s_{j-1}$.
Then $G=H_1+s_j+G_2$.
Since $s_j$ is a kink of $G$,
there are two cases to be considered.

\textbf{Case 1.} The restrictions of $M$ to $H_1$ and $G_2$
are their perfect matchings.
By the induction hypothesis and Corollary \ref{max},
$H_1$ and $G_2$ contain at most $r(B_1)$ and $\sum_{t=2}^{m}r(B_t)$ $M$-alternating face cycles respectively.
If $s_j$ is not $M$-alternating,
then $G$ contains at most $\sum_{t=1}^{m}r(B_{t})$ $M$-alternating face cycles.
If $s_j$ is $M$-alternating,
then the two shared edges of $s_j$ both belong to $M$, and
 the restriction of $M$ to $H_1+s_j$  is also its perfect matching. Since $B_1$ itself is an all-kink decomposition of $H_1+s_j$,
$H_1+s_j$ also contains at most $r(B_1)$
$M$-alternating face cycles.
So $G$ also contains at most $\sum_{t=1}^{m}r(B_{t})$
$M$-alternating face cycles. By Corollary \ref{max},
$Af(G)\leq \sum_{t=1}^{m}r(B_{t})$.

\textbf{Case 2.}  The restrictions of $M$ to $H_1$ and $G_2$
both are not their perfect matchings. So
$s_j$ is $M$-alternating, and none of the  shared edges of $s_j$
 are  in $M$.
Note that $M\triangle s_j$ is a perfect matching
of $G$ such that the two shared edges of $s_j$
both are in $M\triangle s_j$.
So $|\mathcal{A}_{M\triangle s_j}|\geq|\mathcal{A}_{M}|$.
By Theorem \ref{compatibleset},
$af(G,M\triangle s_j)\geq af(G,M)=Af(G)$,
which implies that $af(G,M\triangle s_j)=Af(G)$.
Treating $M\triangle s_j$ as $M$ in Case 1, we  have the required result. 
\end{proof}

An all-kink decomposition of an even polygonal chain can be accomplished in a linear time. So Theorem \ref{maxn} provide
a linear algorithm to find its maximal anti-forcing
number.
Combining Theorems \ref{continuous} with the above linear algorithms to compute the minimum and maximum antiforcing numbers of   even polygonal chains, we have the following conclusion.

\begin{corollary}{\bf .}
The anti-forcing spectrum of an even polygonal chain
can be determined in  linear time.
\end{corollary}

For instance, anti-forcing spectrum  of even polygonal chain $G$ in Fig. \ref{sequence}
is an integer interval $[6,13]$.

\end{document}